\documentclass{amsart}

\usepackage{amssymb}
 
\theoremstyle{plain} 
\newtheorem{theorem}{Theorem}[section] 
\newtheorem{lemma}[theorem]{Lemma} 
\newtheorem{corollary}[theorem]{Corollary} 
\newtheorem{proposition}[theorem]{Proposition} 
\newtheorem{fact}[theorem]{Fact} 

\theoremstyle{definition} 
\newtheorem{definition}[theorem]{Definition} 
\newtheorem{example}[theorem]{Example} 
\newtheorem{problem}{Problem} 
 
\theoremstyle{remark} 
\newtheorem*{remark}{Remark}

\newcommand{\R}{\mathbb{R}} 
 
\newcommand{\I}{\mathcal{I}} 
\newcommand{\J}{\mathcal{J}}
\newcommand{\K}{\mathcal{K}} 
  
\newcommand{\fin}{\textrm{Fin}} 
\DeclareMathOperator{\wlasnosc}{W}
\renewcommand{\subset}{\subseteq}
\renewcommand{\supset}{\supseteq}

\numberwithin{equation}{section}

\makeatletter
\@namedef{subjclassname@2010}{%
  \textup{2010} Mathematics Subject Classification}
\makeatother

\begin{document}
\baselineskip=17pt

\title{Ideal equal Baire classes}

\author{Adam Kwela}
\address{Institute of Mathematics, University of Gda\'{n}sk, ul.~Wita Stwosza 57, 80-952 Gda\'{n}sk, Poland}
\email{adam.kwela@ug.edu.pl}

\author{Marcin Staniszewski}
\address{Institute of Mathematics, University of Gda\'{n}sk, ul.~Wita Stwosza 57, 80-952 Gda\'{n}sk, Poland}
\email{marcin.staniszewski@mat.ug.edu.pl}

\date{\today}



\begin{abstract}
For any Borel ideal we characterize ideal equal Baire system generated by the families of continuous and quasi-continuous functions, i.e., the families of ideal equal limits of sequences of continuous and quasi-continuous functions. 
\end{abstract}

\maketitle

\section{Introduction}

Laczkovich and Rec{\l}aw (see \cite{LR}) and (independently) Debs and Saint Raymond (see \cite{debs-ray}) characterized first Baire class with respect to ideal convergence (the family of pointwise ideal limits of sequences of continuous functions) for every Borel ideal and Polish space. In particular, they characterized Borel ideals for which the first Baire class with respect to ideal convergence is equal to the classical first Baire class. Filip\'{o}w and Szuca (see \cite{filipow-szuca-I-Baire}) have extended this result to ideal discrete convergence and $\left(\I,\fin\right)$-equal convergence. Moreover, they characterized the ideals for which higher Baire classes in the case of all three considered notions of convergence (ideal, ideal discrete and $\left(\I,\fin\right)$-equal convergence) coincide with the classical Baire classes for all perfectly normal topological spaces. In this paper we generalize their results to $\left(\I,\J\right)$-equal convergence. We characterize Baire classes in the case of $\left(\I,\J\right)$-equal convergence for every pair of ideals $\left(\I,\J\right)$, where $\I$ is coanalytic (Theorem \ref{b}).

Recently, Natkaniec and Szuca (see \cite{nat-szuca} and \cite{nat-szuca2}) obtained similar results in the case of quasi-continuous functions instead of continuous functions. Namely, they characterized Baire systems generated by the family of quasi-continuous functions in the case of ideal convergence and ideal discrete convergence for all Borel ideals and metric Baire spaces. In this paper we characterize Baire systems generated by quasi-continuous functions in the case of $\left(\I,\J\right)$-equal convergence for every pair of ideals $\left(\I,\J\right)$, where $\I$ is Borel (Theorem \ref{Borel1}).

One can look at our results from two different points of view. The mentioned characterizations are strictly combinatorial and do not involve any topological notions. Therefore, in some sense we use real analysis to classify pairs of ideals -- we introduce three different q-types and three different c-types of pairs of ideals. The Baire systems generated by continuous (quasi-continuous) functions with respect to ideal equal convergence are the same for all pairs of ideals of the same c-type (q-type). On the other hand, our investigations can be interesting from the point of view of real analysis. All earlier results from this area (cf. \cite{debs-ray}, \cite{filipow-szuca-I-Baire}, \cite{LR}, \cite{nat-szuca} and \cite{nat-szuca2}) have a similar structure: they state that for any $\alpha<\omega_1$ and a Borel ideal the Baire class $\alpha$ (generated by continuous or quasi-continuous functions) with respect to some notion of ideal convergence is equal to one of the Baire classes (generated by the same family of functions) with respect to classical (i.e., not involving ideals) counterpart of the same notion of convergence. We show that the Baire system (generated by continuous or quasi-continuous functions) with respect to ideal equal convergence can be equal to the Baire system (generated by the same family of functions) but with respect to classical convergence (not classical equal convergence). Therefore, the use of ideal equal convergence instead of classical equal convergence can produce new Baire classes. This is the case of the second c-type (for continuous functions) and the second q-type (for quasi-continuous functions). 
 
The paper is organized as follows. Section \ref{Preliminaries} is devoted to introducing necessary notions. In Section \ref{Basic properties of ideal convergence} we collect some basic facts concerning ideal convergence. Finally, Sections \ref{Ideal equal convergence of sequences of quasi-continuous functions} and \ref{Ideal equal convergence of sequences of continuous functions} contain the characterizations of ideal equal Baire classes generated by the families of quasi-continuous and continuous functions, respectively. Both of these sections have a similar structure. We start with introducing some useful notions, then prove partial results and end with the mentioned characterizations (Theorems \ref{Borel1} and \ref{b}) which summarize the considerations included in the whole section.

\section{Preliminaries}
\label{Preliminaries}

We use a standard set-theoretic and topological notation.

\subsection{Ideals}

A collection $\mathcal{I}\subset\mathcal{P}(X)$ is an \emph{ideal on $X$} if it is closed under finite unions and subsets. We additionally assume that each ideal contains $\fin(X)=[X]^{<\omega}$. Hence, we can write $\bigcup\I$ instead of $X$. In this paper we consider only ideals on countable sets. In the theory of ideals a special role is played by the ideal $\fin=\fin(\omega)$. The \emph{filter dual to the ideal $\mathcal{I}$} is the collection $\mathcal{I}^*=\left\{A\subset X:X\setminus A\in\mathcal{I}\right\}$ and $\mathcal{I}^+=\left\{A\subset X:A\notin\mathcal{I}\right\}$ is the collection of all \emph{$\mathcal{I}$-positive sets}. 

An ideal $\I$ is \emph{dense} if every infinite subset of $\bigcup\I$ contains an infinite subset belonging to the ideal. If $Y\subset \bigcup\I$, then the \emph{restriction of $\mathcal{I}$ to the set $Y$}, $\mathcal{I}\upharpoonright Y=\left\{A\cap Y:A\in\mathcal{I}\right\}$, is an ideal on $Y$. We say that a family $\mathcal{G}\subset\mathcal{P}(X)$ \emph{generates the ideal $\mathcal{I}$} if $$\mathcal{I}=\left\{A\subset X:\ \exists_{k\in\omega}\ \exists_{G_0,\ldots,G_k\in\mathcal{G}}\  A\setminus (G_0\cup\ldots\cup G_k)\in\fin(X)\right\}.$$
Ideals $\I$ and $\J$ on $X$ are \emph{orthogonal} if there are $A\in\I$ and $B\in\J$ with $A\cup B=X$. 

The space $2^X$ of all functions $f:X\rightarrow 2$ is equipped with the product topology (each space $2=\left\{0,1\right\}$ carries the discrete topology). We treat $\mathcal{P}(X)$ as the space $2^X$ by identifying subsets of $X$ with their characteristic functions. All topological and descriptive notions in the context of ideals on $X$ will refer to this topology. 

Ideals $\mathcal{I}$ and $\mathcal{J}$ are \emph{isomorphic} if there is a bijection $f:\bigcup\mathcal{J}\rightarrow\bigcup\mathcal{I}$ such that 
$$A\in\mathcal{I} \Leftrightarrow f^{-1}[A]\in\mathcal{J}.$$
Isomorphisms preserve all the properties of ideals considered in this paper. If $\I$ is an ideal on some countable set $X$, then there is always an ideal on $\omega$ isomorphic to it, so it is sufficient to consider only ideals on $\omega$. All the results of this paper, even formulated only for ideals on $\omega$, are true for arbitrary ideals on countable sets.

The structure of ideals on countable sets is often described in terms of orders. We say that \emph{$\mathcal{I}$ is below $\mathcal{J}$ in the Kat\v{e}tov order} ($\mathcal{I}\leq_{K}\mathcal{J}$) if there is $f:\bigcup\mathcal{J}\rightarrow\bigcup\mathcal{I}$ such that
$$A\in\mathcal{I}\Rightarrow f^{-1}[A]\in\mathcal{J}.$$
Furthermore, if $f$ is a bijection between $\bigcup\mathcal{J}$ and $\bigcup\mathcal{I}$, we say that \emph{$\mathcal{J}$ contains an isomorphic copy of $\mathcal{I}$}, and write $\mathcal{I}\sqsubseteq\mathcal{J}$.

Suppose that $\mathcal{I}$ is an ideal on $X$, $A\subset X$ and $(A_n)_{n\in\omega}\subset\mathcal{P}(X)$. Then we define $\mathcal{I}\sqcup A=\{M\cup N:\ M\in\mathcal{I}\wedge N\subset A\}$ and 
$$\mathcal{I}\sqcup (A_n)_{n\in\omega}=\{M\cup N:\ M\in\mathcal{I}\ \wedge\ \exists_{n\in\omega}\ N\subset\bigcup_{i<n}A_i\}.$$

If $X$ and $Y$ are two sets, then their \emph{disjoint sum} is given by $X\oplus Y=\{0\}\times X\cup\{1\}\times Y$. Suppose now that $\I$ and $\J$ are ideals on $X$ and $Y$, respectively. Then we define the ideal $\mathcal{I}\oplus\mathcal{J}$ on $X\oplus Y$ by:
$$A\in\mathcal{I}\oplus\mathcal{J}\ \Leftrightarrow\ \left\{x\in X:\ (0,x)\in A\right\}\in\mathcal{I}\ \wedge\ \left\{y\in Y:\ (1,y)\in A\right\}\in\mathcal{J}.$$
The \emph{product} $\mathcal{I}\otimes\mathcal{J}$ of the ideals $\I$ and $\J$ is an ideal on $X\times Y$ given by:
$$A\in\mathcal{I}\otimes\mathcal{J}\ \Leftrightarrow\ \{x\in X:\ A_x\notin\mathcal{J}\}\in\mathcal{I},$$
where $A_x=\{y\in Y:\ (x,y)\in A\}$. In this definition we allow one of the ideals, $\I$ or $\J$, to contain only the empty set (so we drop the assumption that it contains all finite sets) and in this case we write $\emptyset\otimes\mathcal{J}$ and $\mathcal{I}\otimes\emptyset$ instead of $\{\emptyset\}\otimes\mathcal{J}$ and $\mathcal{I}\otimes\{\emptyset\}$, respectively.

\subsection{Ideal convergence}

Let $\I$ be an ideal on a countable set $I$. A sequence of reals $(x_i)_{i\in I}$ is \emph{$\I$-convergent} to $x\in \R$ if $\{i\in I:\ |x_i-x|\geq \varepsilon\}\in \I$ for any $\varepsilon>0$. In this case we write $(x_i)_{i\in I}\xrightarrow{\I} x$. Similarly, $(x_i)_{i\in I}$ is \emph{$\I$-discretely convergent} to $x$ ($(x_i)_{i\in I}\xrightarrow{\I-d} x$) if we have $\{i\in I:\ x_i\neq x\}\in \I$. A sequence $(f_i)_{i\in I}$ of real-valued functions defined on a set $X$ is \emph{$\I$-pointwise convergent} to $f\in\mathbb{R}^X$ ($(f_i)_{i\in I}\xrightarrow{\I} f$) if $\left(f_i\left(x\right)\right)_{i\in I}\xrightarrow{\I} f\left(x\right)$ for all $x\in X$. Similarly, $(f_i)_{i\in I}$ is \emph{$\I$-discretely convergent} to $f$ ($(f_i)_{i\in I}\xrightarrow{\I-d} f$) if $\left(f_i\left(x\right)\right)_{i\in I}\xrightarrow{\I-d} f\left(x\right)$ for all $x\in X$.

Let now $\I$ and $\J$ be ideals on the same countable set $I$. Let also $(f_i)_{i\in I}\subset\mathbb{R}^X$ and $f\in\mathbb{R}^X$ for some set $X$. We say that $(f_i)_{i\in I}$ is \emph{$\left(\I,\J\right)$-equal convergent} to $f$ ($(f_i)_{i\in I}\xrightarrow{(\I,\J)-e} f$) if there is a sequence $(\varepsilon_i)_{i\in I}$ of positive reals with $(\varepsilon_i)_{i\in I}\xrightarrow{\J}0$ such that $\{i\in I:\ |f_i(x)-f(x)|\geq \varepsilon_i\}\in\I$ for each $x\in X$. In this case we say that \emph{$f$ is an $\left(\I,\J\right)$-equal limit of $(f_i)_{i\in I}$}. If $\I$ and $\J$ are orthogonal ideals and $X$ is a non-empty set, then $\left(\I,\J\right)$-equal limits are not unique (cf. \cite[Theorem 6.1]{fil-stan2}).

The above notions generalize their classical counterparts -- $\fin$-convergence is the classical convergence, $\fin$-discrete convergence is the classical discrete convergence, and $\left(\fin,\fin\right)$-equal convergence is the classical equal convergence (discrete convergence and equal convergence in the classical cases were introduced by Cs{\'a}sz{\'a}r and Laczkovich in \cite{laczkovich-csaszar-equal-convergence-1975}). 

Given two ideals $\I$ and $\J$ on $I$, a set $X$ and a family $\mathcal{F}\subset\mathbb{R}^X$, we denote by $(\mathcal{I},\mathcal{J})\left(\mathcal{F}\right)$ the family of all functions $f\in\mathbb{R}^X$ which can be represented as an $\left(\I,\J\right)$-equal limit of a sequence of functions from $\mathcal{F}$. Moreover, we denote:
\begin{itemize}
	\item $(\mathcal{I},\mathcal{J})_0\left(\mathcal{F}\right)=\mathcal{F}$;
	\item $(\mathcal{I},\mathcal{J})_1\left(\mathcal{F}\right)=(\mathcal{I},\mathcal{J})\left(\mathcal{F}\right)$;
	\item $(\mathcal{I},\mathcal{J})_\alpha\left(\mathcal{F}\right)=(\mathcal{I},\mathcal{J})\left(\bigcup_{\beta<\alpha}(\mathcal{I},\mathcal{J})_\beta\left(\mathcal{F}\right)\right)$.
\end{itemize}

\subsection{Real functions}

Let $X$ be a topological space. By $C\left(X\right)$ we denote the family of all real-valued continuous functions defined on $X$. The class of all functions $f:X\to\R$ with the Baire property is denoted by $\textrm{Baire}\left(X\right)$. By $B_{\alpha}\left(X\right)$ we denote the family of all real-valued functions of Baire class $\alpha$, defined on $X$.

We say that a function $f:X\rightarrow \mathbb{R}$ is \emph{quasi-continuous in $x_{0}\in X$} if for every $\varepsilon>0$ and an open neighbourhood $U$ of $x_{0}$ there exists an open non-empty set $V\subset U$ such that $\left|f\left(x\right)-f\left(x_{0}\right)\right|<\varepsilon$ for every point $x\in V$. A function $f:X\rightarrow \mathbb{R}$ is \emph{quasi-continuous} if it is quasi-continuous in every point $x_{0}\in X$. We denote the class of all quasi-continuous functions on $X$ by $QC\left(X\right)$. All continuous functions as well as all left-continuous and right-continuous functions are quasi-continuous.

A subset $U$ of a topological space $X$ is \emph{semi-open} if $U\subset\overline{{\rm int}U}$. It is known that a function $f\colon X\to\mathbb{R}$ is quasi-continuous if and only if $f^{-1}[U]$ is semi-open for every open set $U\subset\mathbb{R}$. Moreover, a union of any family of semi-open sets is semi-open and an intersection of a semi-open set with an open set is semi-open.

A function $f:X\rightarrow \mathbb{R}$ is \emph{pointwise discontinuous} if the set $C\left(f\right)$ of continuity points of $f$ is dense in $X$. The class of all pointwise discontinuous functions defined on a space $X$ is denoted by $PWD\left(X\right)$. By $C_{q}\left(f\right)$ we denote the set of all quasi-continuity points of $f$. A function $f:X\rightarrow \mathbb{R}$ is in $PWD_0\left(X\right)$ if the set $X\setminus C_{q}\left(f\right)$ is nowhere dense in $X$. 

The notion of quasi-continuity has been introduced by Kempisty (see \cite{kempisty}). The Baire system generated by the family $QC\left(X\right)$ has been described by Grande (see \cite{grande}). Namely, if $X$ is a metric Baire space, then $PWD\left(X\right)$ is the first Baire class generated by $QC\left(X\right)$ with respect to classical convergence, and $PWD_0\left(X\right)$ is the first Baire class generated by $QC\left(X\right)$ with respect to discrete convergence. All higher Baire classes in both cases are equal to $\textrm{Baire}\left(X\right)$.

\section{Basic properties of ideal convergence}
\label{Basic properties of ideal convergence}

In this section we collect some basic observations which will be useful in our further considerations.

\begin{lemma}[Natkaniec and Szuca, {\cite[Corollary 14]{nat-szuca}}]
\label{0}
Suppose that $\mathcal{I}$ is an analytic ideal on $\omega$ and $X$ is a topological space. If $(f_n)_{n\in\omega}\subset \textrm{Baire} (X)$ is $\I$-convergent to some $f\colon X\to\mathbb{R}$, then $f\in \textrm{Baire} (X)$. 
\end{lemma}

\begin{lemma}
\label{1}
Suppose that $\mathcal{I}$ is an analytic (coanalytic) ideal on $\omega$. Then $\mathcal{I}\sqcup A$ and $\mathcal{I}\sqcup (A_n)_{n\in\omega}$ are analytic (coanalytic) for any $A\subset \omega$ and $(A_n)_{n\in\omega}\subset\mathcal{P}(\omega)$.
\end{lemma}

\begin{proof}
Let $\varphi\colon\mathcal{P}(\omega)\to\mathcal{P}(\omega)$ be given by $\varphi(M)=M\setminus A$. For each $n\in\omega$ let also $\varphi_n\colon\mathcal{P}(\omega)\to\mathcal{P}(\omega)$ be given by $\varphi_n(M)=M\setminus \bigcup_{i<n}A_i$. Then we have $\mathcal{I}\sqcup A=\varphi^{-1}[\I]$ and $\mathcal{I}\sqcup (A_n)_{n\in\omega}=\bigcup_{n\in\omega}\varphi_n^{-1}[\I]$. Now it suffices to observe that $\varphi$ as well as all $\varphi_n$'s are continuous.
\end{proof}

Let $\I,\J$ be ideals on $\omega$. By $\wlasnosc(\I,\J)$ we denote the following sentence: For every partition $(A_n)_{n\in\omega}\subset\J$ of $\omega$ there exists $S\notin\I$ such that $A_n\cap S\in\I$ for every $n\in\omega$.

\begin{lemma}[Filip{\'o}w and Staniszewski, {\cite[Theorem 5.2]{fil-stan2}}]\label{wlasn}
Let $\I,\J$ be ideals on $\omega$ such that $\wlasnosc(\I,\J) $ does not hold.
For every set $X$ and every sequence  $(f_n)_{n\in \omega}$ of real-valued functions defined on $X$,
if $(f_n)_{n\in \omega}\xrightarrow{\I} f$ for some $f\in\mathbb{R}^X$, then 
 $(f_n)_{n\in \omega}\xrightarrow{\left(\I,\J\right)-e}f$.
\end{lemma}

\begin{remark}
Suppose that $\I$ is an ideal on $\omega$. Topological spaces $X$ such that for every sequence of real-valued continuous functions $(f_n)_{n\in \omega}$ defined on $X$,
if $(f_n)_{n\in \omega}\xrightarrow{\fin} 0$, then 
 $(f_n)_{n\in \omega}\xrightarrow{\left(\I,\I\right)-e}0$, are called $\I QN$-spaces. Recently, {\v{S}}upina (see \cite{supina}) showed that an ideal $\I$ contains an isomorphic copy of the ideal $\fin\otimes\fin$ if and only if every topological space is an $\I QN$-space.
\end{remark}

\begin{lemma}
\label{2}
Suppose that $\mathcal{I}$ is an ideal on $\omega$ and $(f_n)_{n\in\omega}\subset\mathbb{R}^X$. If $(f_n)_{n\in\omega}\xrightarrow{\mathcal{I}-d} f$ for some $f\in\mathbb{R}^X$, then $(f_n)_{n\in\omega}\xrightarrow{(\mathcal{I},\mathcal{J})-e} f$ for any ideal $\mathcal{J}$.
\end{lemma}

\begin{proof}
Let $\varepsilon_n=\frac{1}{n+1}$ for each $n\in\omega$. Then $(\varepsilon_n)_{n\in\omega}\xrightarrow{\mathcal{J}} 0$ for any ideal $\mathcal{J}$ and we have $$\left\{n\in\omega:\ |f(x)-f_n(x)|\geq \frac{1}{n+1}\right\}\subset \{n\in\omega:\ f(x)\neq f_n(x)\}\in\mathcal{I}$$ for any $x\in X$.
\end{proof}

\begin{lemma}
\label{3}
Suppose that $\mathcal{I}$ and $\mathcal{J}$ are ideals on $\omega$, $(f_n)_{n\in\omega}\subset\mathbb{R}^X$ and $f\in\mathbb{R}^X$ for some set $X$. If $(f_n)_{n\in\omega}\xrightarrow{(\mathcal{I},\mathcal{J})-e} f$ and $(\varepsilon_n)_{n\in\omega}$ is the sequence of positive reals $\mathcal{J}$-convergent to $0$ from the definition of $(\mathcal{I},\mathcal{J})$-equal convergence, then $(f_n)_{n\in\omega}\xrightarrow{\mathcal{I}\sqcup (A_k)_{k\in\omega}} f$, where $A_0=\{n\in\omega:\varepsilon_n\geq 1\}\in\mathcal{J}$ and $A_k=\{n\in\omega:\ \frac{1}{k+1}\leq \varepsilon_n<\frac{1}{k}\}\in\mathcal{J}$ for all $k\geq 1$.
\end{lemma}

\begin{proof}
We will show that $(f_n)_{n\in\omega}\xrightarrow{\mathcal{I}\sqcup (A_n)_{n\in\omega}} f$. Consider any $x\in X$ and $\varepsilon>0$. There is $k\in\omega$ with $\frac{1}{k}<\varepsilon$. Then $\{n\in\omega:\ |f_n(x)-f(x)|\geq \varepsilon\}$ is contained in
$$\bigcup_{i< k}A_i\ \cup\ \left\{n\in\bigcup_{i\geq k}A_i:\ |f_n(x)-f(x)|\geq\varepsilon_n\right\}\in \mathcal{I}\sqcup (A_n)_{n\in\omega}.$$
This finishes the proof.
\end{proof}

\begin{lemma}
Suppose that $(f_n)_{n\in\omega}\subset\mathbb{R}^X$ for some set $X$. Let $\I_1, \I_2, \J_1$ and $\J_2$ be ideals on $\omega$ such that $\I_1\subset\I_2$ and $\J_1\subset\J_2$. If $(f_n)_{n\in\omega}\xrightarrow{(\I_1,\J_1)-e} f$ for some $f\in\mathbb{R}^X$, then $(f_n)_{n\in\omega}\xrightarrow{(\I_2,\J_2)-e} f$.
\end{lemma}

\begin{proof}
Straightforward.
\end{proof}

\begin{lemma}
\label{4}
Suppose that $\mathcal{F}\subset\mathbb{R}^X$ for some set $X$. Let $\I_1, \I_2, \J_1$ and $\J_2$ be ideals on $\omega$. Then $(\I_1\oplus \I_2,\J_1\oplus \J_2)\left(\mathcal{F}\right)=(\mathcal{I}_1,\mathcal{J}_1)\left(\mathcal{F}\right)\cap (\mathcal{I}_2,\mathcal{J}_2)\left(\mathcal{F}\right)$.
\end{lemma}

\begin{proof}
Take any $f\in (\I_1\oplus \I_2,\J_1\oplus \J_2)\left(\mathcal{F}\right)$. There are a sequence of real-valued functions $(f_{(i,n)})_{(i,n)\in 2\times\omega}\subset\mathbb{R}^X$ and a sequence $(\varepsilon_{(i,n)})_{(i,n)\in 2\times\omega}$ of positive reals $\left(\J_1\oplus \J_2\right)$-convergent to $0$ such that $\{(i,n)\in 2\times\omega: |f_{(i,n)}(x)-f(x)|\geq \varepsilon_{(i,n)}\}\in\I_1\oplus \I_2$ for each $x\in X$. Then $(\varepsilon_{(0,n)})_{n\in\omega}$ is $\J_1$-convergent to $0$, $(\varepsilon_{(1,n)})_{n\in\omega}$ is $\J_2$-convergent to $0$ and for each $x\in X$ we have $\{n\in\omega: |f_{(0,n)}(x)-f(x)|\geq \varepsilon_{(0,n)}\}\in\mathcal{I}_1$ and $\{n\in\omega: |f_{(1,n)}(x)-f(x)|\geq \varepsilon_{(1,n)}\}\in\mathcal{I}_2$. Therefore, $f\in (\mathcal{I}_1,\mathcal{J}_1)\left(\mathcal{F}\right)\cap (\mathcal{I}_2,\mathcal{J}_2)\left(\mathcal{F}\right)$.

To show the opposite inclusion, take any $f\in (\mathcal{I}_1,\mathcal{J}_1)\left(\mathcal{F}\right)\cap (\mathcal{I}_2,\mathcal{J}_2)\left(\mathcal{F}\right)$. There are $(f^1_n)_{n\in\omega},(f^2_n)_{n\in\omega}\subset\mathbb{R}^X$ and two sequences of positive reals $(\varepsilon^1_n)_{n\in\omega}$ and $(\varepsilon^2_n)_{n\in\omega}$ $\J_1$-convergent to $0$ and $\J_2$-convergent to $0$, respectively, such that for each $x\in X$ we have $\{n\in\omega: |f^1_n(x)-f(x)|\geq \varepsilon_{n}^{1}\}\in\mathcal{I}_1$ and $\{n\in\omega: |f^2_n(x)-f(x)|\geq \varepsilon_{n}^{2}\}\in\mathcal{I}_2$. Define $\varepsilon_{(i,n)}=\varepsilon^{i+1}_n$ and $f_{(i,n)}=f^{i+1}_n$ for each $(i,n)\in 2\times\omega$. Then $(\varepsilon_{(i,n)})_{(i,n)\in 2\times\omega}$ is $\J_1\oplus \J_2$-convergent to $0$. Moreover, given any $x\in X$ we have $\{(i,n)\in 2\times\omega: |f_{(i,n)}(x)-f(x)|\geq \varepsilon_{(i,n)}\}\in\I_1\oplus\I_2$. This finishes the proof.
\end{proof}

Recall that if $\I$ and $\J$ are orthogonal ideals and $X$ is non-empty, then $\left(\I,\J\right)$-equal limits are not unique.

\begin{lemma}
\label{5}
If $\I$ and $\J$ are orthogonal ideals on $\omega$, then $(\mathcal{I},\mathcal{J})\left(\mathcal{F}\right)=\mathbb{R}^X$ for any set $X$ and non-empty family of functions $\mathcal{F}\subset\mathbb{R}^X$.
\end{lemma}

\begin{proof}
Let $A\in\I$ and $B\in\J$ be such that $A\cup B=\omega$. By Lemma \ref{4} we have $(\mathcal{I},\mathcal{J})\left(\mathcal{F}\right)=(\mathcal{P}(A),\mathcal{J}\upharpoonright A)\left(\mathcal{F}\right)\cap (\mathcal{I}\upharpoonright B,\mathcal{P}(B))\left(\mathcal{F}\right)$. Let $g\in\mathcal{F}$.

Firstly, we will show that $(\mathcal{P}(A),\mathcal{J}\upharpoonright A)\left(\mathcal{F}\right)\supset\mathbb{R}^X$ (the other inclusion is trivial). Take any $f\in\mathbb{R}^X$ and define $\varepsilon_n=\frac{1}{n+1}$ and $f_n=g$ for all $n\in A$. Then $(\varepsilon_n)_{n\in A}$ is $\mathcal{J}\upharpoonright A$-convergent to $0$ and we have  $\{n\in A: |f_n(x)-f(x)|\geq \varepsilon_n\}\in\mathcal{P}(A)$ for any $x\in X$.

Now we deal with the inclusion $(\mathcal{I}\upharpoonright B,\mathcal{P}(B))\left(\mathcal{F}\right)\supset\mathbb{R}^X$. Take any $f\in\mathbb{R}^X$ and define $\varepsilon_n=n$ and $f_n=g$ for all $n\in B$. Then $(\varepsilon_n)_{n\in B}$ is $\mathcal{P}(B)$-convergent to $0$. Moreover, given any $x\in X$, there are only finitely many $n\in B$ with $|f_n(x)-f(x)|\geq n$. Hence, $\{n\in B:\ |f_n(x)-f(x)|\geq \varepsilon_n\}\in\fin\subset\mathcal{I}\upharpoonright B$ for any $x\in X$.
\end{proof}

\section{Ideal equal convergence of sequences of quasi-continuous functions}
\label{Ideal equal convergence of sequences of quasi-continuous functions}

In this section we want to characterize ideal equal Baire classes generated by the family of quasi-continuous functions. In the first subsection we introduce some useful notions. Next, we give some examples and prove the mentioned characterization.

\subsection{An infinite game and the q-types}

Let $\I$ be an ideal. Laflamme (see \cite{laf}) defined an infinite game $G_{1}\left(\I\right)$ as follows: Player I in his $n$'th move plays an element $C_{n}\in\I$, and then Player II responses with any $a_{n}\notin C_{n}$. Player I wins if $\left\{a_{n}:\ n\in \omega\right\}\in \I$. Otherwise, Player II wins.

\begin{theorem}[{\cite[Fact 3.10]{KZ}}, see also {\cite[Section 5]{KS}}]
\label{determinacja}
If $\mathcal{I}$ is a coanalytic ideal, then the game $G_{1}(\mathcal{I})$ is determined, i.e., one of the players has a winning strategy.
\end{theorem}

An ideal $\mathcal{I}$ is called \emph{$\omega$-$+$-diagonalizable} if there is a countable family $(X_n)_{n\in\omega}\subseteq\mathcal{I}^+$ such that for each $Y\in\mathcal{I}^*$ there is $n\in\omega$ with $X_n\subset Y$ (see \cite{laf}). An ideal $\mathcal{I}$ on $\omega$ is \emph{weakly Ramsey} if for every coloring $f:[\omega]^2\rightarrow 2$, such that for each $x\in \omega$ either $\left\{y\in \omega:\ f\left(\left\{x,y\right\}\right)=0\right\}\in\mathcal{I}$ or $\left\{y\in \omega:\ f\left(\left\{x,y\right\}\right)=1\right\}\in\mathcal{I}$, there is an $\mathcal{I}$-positive $H$ with $f\upharpoonright[H]^2$ constant (this notion was introduced in \cite{laf} in a slightly different way -- the equivalence of the definition from \cite{laf} with the presented one is proved in \cite{K}).

\begin{fact}
\label{monotonicznosc}
The following hold.
\begin{enumerate}
	\item If an ideal $\I$ is $\omega$-$+$-diagonalizable, then so is any ideal $\J\subset\I$. 
	\item If an ideal $\I$ is not weakly Ramsey, then so is any ideal $\J\supset\I$. 
	\item If $\I$ is not weakly Ramsey, then so is $\I\upharpoonright A$ for any $A$.
\end{enumerate}
\end{fact}

\begin{proof}
Straightforward.
\end{proof}

Laflamme introduced the notions of $\omega$-$+$-diagonalizability and weak Ramseyness in order to give the following characterization.

\begin{theorem}[Laflamme, {\cite[Theorem 2.7]{laf}}]\label{kh777}
Let $\mathcal{I}$ be an ideal.
\begin{enumerate}
	\item Player I has a winning strategy in $G_{1}(\mathcal{I})$ if and only if the ideal $\mathcal{I}$ is not weakly Ramsey.	
	\item Player II has a winning strategy in $G_{1}(\mathcal{I})$ if and only if the ideal $\mathcal{I}$ is $\omega$-$+$-diagonalizable.
\end{enumerate}
\end{theorem}

It follows from the above two theorems that any coanalytic ideal either is not weakly Ramsey or is $\omega$-$+$-diagonalizable.

$\mathcal{WR}$ is an ideal on $\omega\times\omega$ generated by vertical lines, i.e., sets of the form $\{n\}\times\omega$ for $n\in\omega$ (which we call \emph{generators of the first type}) and sets $G$ such that for every $(i,j),(k,l)\in G$ either $i>k+l$ or $k>i+j$ (which we call \emph{generators of the second type}). 

\begin{theorem}[Kwela, {\cite[Theorem 1.3]{K}}]
\label{WR}
The following are equivalent for any ideal $\mathcal{I}$ on $\omega$:
\begin{enumerate}
\item $\mathcal{I}$ is not weakly Ramsey;
\item $\mathcal{WR}\sqsubseteq\mathcal{I}$;
\item $\mathcal{WR}\leq_{K}\mathcal{I}$.
\end{enumerate}
\end{theorem}

\begin{fact}
\label{dense}
Each ideal which is not dense, has to be weakly Ramsey and $\omega$-$+$-diagonalizable. 
\end{fact}

\begin{proof}
The first statement follows from Theorem \ref{WR} and the fact that the ideal $\mathcal{WR}$ is dense (cf. \cite[Lemma 5.3]{K}). To show the second one, take any ideal $\I$ on $X$ which is not dense and let $A$ be such that $\I\upharpoonright A$ is isomorphic to $\fin$. Then $(A\setminus n)_{n\in\omega}$ is the family $\omega$-$+$-diagonalizing $\I$.
\end{proof}

We are ready to define q-types of pairs of ideals.

\begin{definition}
Let $\I$ and $\J$ be ideals.
\begin{enumerate}
	\item $(\I,\J)$ is of the \emph{first q-type} if for any sequence $(A_n)_{n\in\omega}$ of elements of $\J$ the ideal $\mathcal{I}\sqcup (A_n)_{n\in\omega}$ is $\omega$-$+$-diagonalizable.
	\item $(\I,\J)$ is of the \emph{second q-type} if there is a sequence $(A_n)_{n\in\omega}$ of elements of $\J$ such that the ideal $\mathcal{I}\sqcup (A_n)_{n\in\omega}$ is not weakly Ramsey, but for any $A\in\J$ the ideal $\mathcal{I}\sqcup A$ is $\omega$-$+$-diagonalizable.
	\item $(\I,\J)$ is of the \emph{third q-type} if there is $A\in\J$ such that the ideal $\mathcal{I}\sqcup A$ is not weakly Ramsey.
\end{enumerate}
\end{definition}

\begin{fact}
\label{det}
If $\I$ is coanalytic, then each pair $(\I,\J)$ is of some q-type.
\end{fact}

\begin{proof}
It follows from Lemma \ref{1} that $\mathcal{I}\sqcup (A_n)_{n\in\omega}$ and $\mathcal{I}\sqcup A$ are coanalytic for any $(A_n)_{n\in\omega}$ and $A\subset\omega$. Then we are done by Theorems \ref{determinacja} and \ref{kh777}.
\end{proof}

\subsection{Examples}

In this subsection we give examples of pairs of ideals for every q-type. Moreover, we investigate how the ideal $\I$ can determine the q-type of the pair $(\I,\J)$.

First example shows that there is a pair $(\I,\J)$ of the second q-type and that $(\I,\J)$ and $(\J,\I)$ can be of different q-types.

\begin{example}
Let $\I=\emptyset\otimes\fin$ and $\J=\fin\otimes\emptyset$. Then $(\I,\J)$ is of the second q-type. Indeed, $\mathcal{I}\sqcup (\{n\}\times\omega)_{n\in\omega}=\fin\otimes\fin$ and $\mathcal{WR}\sqsubseteq\fin\otimes\fin$. On the other hand, $\mathcal{I}\sqcup A$ is not dense for any $A\in\J$, so it has to be $\omega$-$+$-diagonalizable by Fact \ref{dense}. 

Note also that $(\J,\I)$ is of the first q-type. Indeed, it follows from Fact \ref{dense}, since $\mathcal{J}\sqcup (A_n)_{n\in\omega}$ is not dense for any $(A_n)_{n\in\omega}\subset\I$.
\end{example}

\begin{fact}
\label{zawieranie}
Suppose that $\J\subset\I$. Then:
\begin{itemize}
	\item $(\I,\J)$ is of the first q-type if and only if $\I$ is $\omega$-$+$-diagonalizable;
	\item $(\I,\J)$ is not of the second q-type for any $\J$;
	\item $(\I,\J)$ is of the third q-type if and only if $\I$ is not weakly Ramsey.
\end{itemize}
\end{fact}

\begin{proof}
Straightforward.
\end{proof}

The following example shows that in some cases the q-type of the pair $(\I,\J)$ depends only on $\I$.

\begin{example}
The ideal $\fin$ is such that for any proper ideal $\J$ on $\omega$ (i.e., an ideal which is a proper subset of $\mathcal{P}(\omega)$) the pair $(\fin,\J)$ is of the first q-type (by Fact \ref{dense}, since $\fin\sqcup (A_n)_{n\in\omega}$ is not dense for any $(A_n)_{n\in\omega}\subset\J$). On the other hand, the ideal $\mathcal{WR}$ is such that for any ideal $\J$ (not necessarily proper) the pair $(\mathcal{WR},\J)$ is of the third q-type. 
\end{example}

Now we show that there is no ideal $\I$ such that the pair $(\I,\J)$ is of the second q-type, whatever the ideal $\J$ is like.

\begin{remark}
If $\I$ is an ideal on $I$ such that there is some $\J$ with $(\I,\J)$ of the second q-type, then $(\I,\fin (I))$ is of the first q-type. Therefore, there is no ideal $\I$ such that for any $\J$ the pair $(\I,\J)$ is of the second q-type. Also, there is no $\I$ such that one can find $\J_1$ and $\J_2$ with $(\I,\J_1)$ of the second q-type and $(\I,\J_2)$ of the third q-type, but for any $\J$ the pair $(\I,\J)$ is not of the first q-type.
\end{remark}

\begin{proof}
Straightforward.
\end{proof}

Next examples show that there are ideals $\I$ such that the q-type of the pair $(\I,\J)$ depends on $\J$. We omit detailed arguments, since they are similar to the ones already used in this subsection.

\begin{example}
Consider $\I=\fin\oplus\mathcal{WR}$. 
\begin{itemize}
	\item If $\J=\fin (\omega\oplus\omega^2)$, then $(\I,\J)$ is of the first q-type;
	\item $(\I,\J)$ is not of the second q-type for any $\J$;
	\item if $\J=\mathcal{P}(\omega)\oplus\fin (\omega^2)$, then $(\I,\J)$ is of the third q-type.
\end{itemize}
\end{example}

\begin{example}
Consider $\I=\left(\emptyset\otimes\fin\right)\oplus\mathcal{WR}$.
\begin{itemize}
	\item If $\J=\fin (2\times\omega^2)$, then $(\I,\J)$ is of the first q-type;
	\item if $\J=\left(\fin\otimes\emptyset\right)\oplus\fin (\omega^2)$, then $(\I,\J)$ is of the second q-type;
	\item if $\J=\mathcal{P}(\omega^2)\oplus\fin (\omega^2)$, then $(\I,\J)$ is of the third q-type.
\end{itemize}
\end{example}

\subsection{The first and third q-type}

In this subsection we characterize $(\mathcal{I},\mathcal{J})\left(QC\left(X\right)\right)$ for all pairs of ideals $(\mathcal{I},\mathcal{J})$ of the first or third q-type.

\begin{proposition}
\label{q-type 1}
Suppose that $X$ is a metric Baire space, $\I$ and $\J$ are ideals on $\omega$ and $(\mathcal{I},\mathcal{J})$ is of the first q-type. Then $(\mathcal{I},\mathcal{J})\left(QC\left(X\right)\right)= PWD_0\left(X\right)$.
\end{proposition}

\begin{proof}
By \cite[Theorem 9]{nat-szuca2}, every $f\in PWD_0\left(X\right)$ is an $\mathcal{I}$-discrete limit of a sequence of quasi-continuous functions. Then, by Lemma \ref{2}, $f\in (\mathcal{I},\mathcal{J})\left(QC\left(X\right)\right)$.

Now we will show that $(\mathcal{I},\mathcal{J})\left(QC\left(X\right)\right)\subset PWD_0\left(X\right)$. Fix a sequence $(f_n)_{n\in\omega}$ of quasi-continuous functions such that $(f_n)_{n\in\omega}\xrightarrow{(\mathcal{I},\mathcal{J})-e} f$ for some $f\in \R^{X}$. Let $(\varepsilon_n)_{n\in\omega}$ be the sequence of positive reals $\mathcal{J}$-convergent to $0$ from the definition of $(f_n)_{n\in\omega}\xrightarrow{(\mathcal{I},\mathcal{J})-e} f$. Then, by Lemma \ref{3}, we have $(f_n)_{n\in\omega}\xrightarrow{\mathcal{I}\sqcup (A_k)_{k\in\omega}} f$, where $A_0=\{n\in\omega:\varepsilon_n\geq 1\}\in\mathcal{J}$ and $A_k=\{n\in\omega:\ \frac{1}{k+1}\leq \varepsilon_n<\frac{1}{k}\}\in\mathcal{J}$ for all $k\geq 1$. Since $(\mathcal{I},\mathcal{J})$ is of the first q-type, $\mathcal{I}\sqcup (A_n)_{n\in\omega}$ is $\omega$-$+$-diagonalizable, so $f\in PWD\left(X\right)$ by \cite[Proposition 3.1]{nat-szuca}. Therefore, the set $C\left(f\right)$ is residual in $X$. Let $(D_n)_{n\in\omega}\subset (\mathcal{I}\sqcup (A_n)_{n\in\omega})^+$ be the family $\omega$-$+$-diagonalizing $\mathcal{I}\sqcup (A_n)_{n\in\omega}$.

We will show that $f$ is in $PWD_0\left(X\right)$, i.e., that $X\setminus C_q(f)$ is nowhere dense. Consider any open and non-empty set $U\subset X$. Since $(f_n)_{n\in\omega}\xrightarrow{(\mathcal{I},\mathcal{J})-e} f$, for every $x\in C\left(f\right)\cap U$ there is $n_{x}$ with $|f_n(x)-f(x)|< \varepsilon_n$ for every $n\in D_{n_{x}}$. Since $X$ is a Baire space, there exists $m\in \omega$ such that the set $C=\left\{x\in C\left(f\right)\cap U:\ n_{x}=m\right\}$ is dense in some open non-empty set $U_{0}\subseteq U$. We have $|f_i(x)-f(x)|< \varepsilon_i$ for every $x\in C$ and every $i\in D_{m}$. Now it is enough to show that $f$ is quasi-continuous in every point from $U_{0}$. 

Fix $x_{0}\in U_{0}$, $\varepsilon>0$ and an open non-empty set $W$ such that $x_{0}\in W$. Without loss of generality we can assume that $W\subseteq U_{0}$. There exists $F\in\mathcal{I}^*\subset(\mathcal{I}\sqcup (A_n)_{n\in\omega})^*$ such that $|f_i(x_{0})-f(x_{0})|< \varepsilon_i$ for every $i\in F$. The set $F\cap D_{m}$ does not belong to $\mathcal{I}\sqcup (A_n)_{n\in\omega}$. In particular, it intersects infinitely many $A_n$'s, so there exists $n\in F\cap D_{m}$ such that $\varepsilon_n<\frac{\varepsilon}{4}$. By quasi-continuity of $f_n$, there exists $t\in W\cap C$ such that $|f_n(t)-f_n(x_{0})|< \frac{\varepsilon}{4}$. Since $f$ is continuous in $t$, there is an open non-empty set $V\subseteq W$ such that $|f(x)-f(t)|< \frac{\varepsilon}{4}$ for every $x\in V$. Then 
$$|f(x)-f(x_{0})|\leq |f(x)-f(t)|+|f(t)-f_{n}(t)|+|f_{n}(t)-f_{n}(x_{0})|+|f_{n}(x_{0})-f(x_{0})|<\varepsilon$$
for every $x\in V$. Therefore, $f$ is quasi-continuous in $x_{0}$.
\end{proof}

\begin{proposition}
\label{q-type 3}
Suppose that $X$ is a metric Baire space, $\I$ and $\J$ are ideals on $\omega$ and $(\mathcal{I},\mathcal{J})$ is of the third q-type. Then $\textrm{Baire}\left(X\right)\subset (\mathcal{I},\mathcal{J})\left(QC\left(X\right)\right)$.
\end{proposition}

\begin{proof}
Since $(\mathcal{I},\mathcal{J})$ is of the third q-type, there is $A\in\J$ such that $\mathcal{I}\sqcup A$ is not weakly Ramsey. Note that in particular $\mathcal{I}\upharpoonright (\omega\setminus A)$ is not weakly Ramsey (by Fact \ref{monotonicznosc}), and 
$$(\mathcal{I},\mathcal{J})\left(QC\left(X\right)\right)=(\mathcal{I}\upharpoonright (\omega\setminus A),\mathcal{J}\upharpoonright (\omega\setminus A))\left(QC\left(X\right)\right)\cap (\mathcal{I}\upharpoonright A,\mathcal{P}(A))\left(QC\left(X\right)\right)$$ 
by Lemma \ref{4} (we assume that $A$ and $\omega\setminus A$ both are infinite -- otherwise, it suffices to consider only one of the intersected families of functions). Since the ideals $\mathcal{I}\upharpoonright A$ and $\mathcal{P}(A)$ are orthogonal, we have $\textrm{Baire}\left(X\right)\subset\mathbb{R}^X\subset (\mathcal{I}\upharpoonright A,\mathcal{P}(A))\left(QC\left(X\right)\right)$ by Lemma \ref{5}. Therefore, it suffices to show that $\textrm{Baire}\left(X\right)\subset (\mathcal{I}\upharpoonright (\omega\setminus A),\mathcal{J}\upharpoonright (\omega\setminus A))\left(QC\left(X\right)\right)$.

Let $f\colon X\to\mathbb{R}$ be a function possessing the Baire property. By Theorem \ref{WR} and \cite[Proposition 16]{nat-szuca2}, there is a sequence $(g_{(n,m)})_{(n,m)\in\omega^2}$ of quasi-continuous functions $\mathcal{WR}$-discretely convergent to $f$. By Theorem \ref{WR}, there also is a bijection $\pi\colon\omega\setminus A\to\omega^2$ with $\pi^{-1}[M]\in\mathcal{I}\upharpoonright (\omega\setminus A)$ for each $M\in\mathcal{WR}$. Define $f_n=g_{\pi(n)}$ for all $n\in\omega\setminus A$. Then each $f_n$ is quasi-continuous and $(f_n)_{n\in\omega\setminus A}\xrightarrow{\mathcal{I}-d} f$, so from Lemma \ref{2} we obtain that $(f_n)_{n\in\omega\setminus A}\xrightarrow{(\mathcal{I},\mathcal{J})-e} f$.

\end{proof}

\subsection{The second q-type}

In this subsection we characterize $(\mathcal{I},\mathcal{J})\left(QC\left(X\right)\right)$ for all pairs of ideals $(\mathcal{I},\mathcal{J})$ of the second q-type.

\begin{proposition}
\label{q-type 2.1}
Suppose that $X$ is a Baire space, $\I$ and $\J$ are ideals on $\omega$ such that $\mathcal{I}\sqcup A$ is $\omega$-$+$-diagonalizable for any $A\in\J$. Then $(\mathcal{I},\mathcal{J})\left(QC\left(X\right)\right)\subset PWD\left(X\right)$.
\end{proposition}

\begin{proof}
This proof is based on the proof of \cite[Proposition 3.1]{nat-szuca}.

Fix a sequence $(f_n)_{n\in\omega}$ of quasi-continuous functions such that $(f_n)_{n\in\omega}\xrightarrow{(\mathcal{I},\mathcal{J})-e} f$ for some $f\in \R^{X}$. Let $(\varepsilon_n)_{n\in\omega}$ be the sequence of positive reals $\mathcal{J}$-convergent to $0$ from the definition of $(\mathcal{I},\mathcal{J})$-equal convergence. 

Suppose that $f$ is not pointwise discontinuous. By \cite[Lemma 2.1(1)]{nat-szuca}, there are reals $\alpha<\beta$ and an open non-empty set $U\subset X$ such that $E=f^{-1}[(-\infty,\alpha)]$ and $F=f^{-1}[(\beta,+\infty)]$ are both dense in $U$. By shrinking $U$, without loss of generality we can assume that $E\cap W$ is not meager for every open non-empty $W\subset U$. Let $\varepsilon=\frac{\beta-\alpha}{2}$ and $A=\{n\in\omega:\ \varepsilon_n\geq\varepsilon\}\in\J$. Let $(D_n)_{n\in\omega}\subset (\mathcal{I}\sqcup A)^+$ be the family $\omega$-$+$-diagonalizing $\mathcal{I}\sqcup A$. 

For each $x\in U\cap E$ there is $n_{x}$ with $|f_i(x)-f(x)|<\varepsilon_i$ for every $i\in D_{n_{x}}$. Note that $f_i(x)<\alpha+\varepsilon_i$ for every $x\in U\cap E$ and $i\in D_{n_{x}}$. Since $X$ is a Baire space, there exists $m\in \omega$ such that the set $\left\{x\in U\cap E:\ n_{x}=m\right\}$ is dense in some open and non-empty set $W\subseteq U$. Recall that each $f_i$ is quasi-continuous. Therefore, for every $i\in D_{m}\setminus A$ we have $f_i(x)<\alpha+\varepsilon$ for all $x\in W$ (apply the definition of quasi-continuity to $x$, $W$ and $\varepsilon-\varepsilon_i$). 

On the other hand, take any $x_0\in W\cap F$ and note that 
$$C=\{i\in\omega:\ |f_i(x_{0})-f(x_{0})|<\varepsilon_i\ \wedge\ \varepsilon_i<\varepsilon\}\in (\mathcal{I}\sqcup A)^*.$$ 
Hence, there is some $i_0\in D_m\cap C$ and we obtain that $f_{i_0}(x_0)>\beta-\varepsilon=\alpha+\varepsilon$. A~contradiction.
\end{proof}

Now we want to show that $(\mathcal{I},\mathcal{J})\left(QC\left(X\right)\right)\supset PWD\left(X\right)$ for any metric Baire space $X$ provided that $(\mathcal{I},\mathcal{J})$ is of the second q-type. This is the most technical part of our considerations. We will need some lemmas.

\begin{lemma}
\label{lem1}
Let $X$ be a topological space and $f\in\mathbb{R}^X$ be pointwise discontinuous. Then for every $\varepsilon>0$ there are a closed nowhere dense set $N$ and a continuous function $g\colon X\setminus N\to\mathbb{R}$, such that $|f(x)-g(x)|<\varepsilon$ for all $x\in X\setminus N$.
\end{lemma}

\begin{proof}
We will use the Zorn's lemma. Fix $\varepsilon>0$ and let $\mathbb{P}$ be the family of all pairs $(U,h)$ such that $U$ is an open subset of $X$ and $h\colon U\to\mathbb{R}$ is a continuous function satisfying $|f(x)-h(x)|<\varepsilon$ for all $x\in U$. Observe that $\mathbb{P}$ is non-empty. Indeed, take any $y\in C(f)$. Then there is an open set $U$ containing $y$, such that $|f(y)-f(x)|<\varepsilon$ for any $x\in U$. Define $h\colon U\to\mathbb{R}$ by $h(x)=f(y)$ for all $x\in U$. Then $(U,h)\in\mathbb{P}$.

The order is defined as follows:
$$(U,h)\preceq (U',h') \Longleftrightarrow U\subseteq U'\ \wedge\ h\subseteq h'.$$
It is easy to check that this is a partial order on $\mathbb{P}$. Moreover, if $((U_\alpha,h_\alpha))_{\alpha<\kappa}$ is a chain in $(\mathbb{P},\preceq)$, then $(\bigcup_{\alpha<\kappa}U_\alpha,\bigcup_{\alpha<\kappa}h_\alpha)$ is its upper bound.

By Zorn's lemma, there is some $(U,g)$ maximal in $(\mathbb{P},\preceq)$. It suffices to show that $N=X\setminus U$ is nowhere dense. Suppose otherwise. Then there is an open non-empty $V\subset N$. Take any $y\in V\cap C(f)$. There is an open set $W$ containing $y$, such that $|f(y)-f(x)|<\varepsilon$ for any $x\in W$. Let $U'=U\cup W$ and $g'\colon U'\to\mathbb{R}$ be given by $g'(x)=g(x)$ for $x\in U$ and $g'(x)=f(y)$ for $x\in W$ (recall that $U$ and $W$ are disjoint). Then $g'$ is continuous and $(U,g)\prec (U',g')$. A contradiction with maximality of $(U,g)$. 
\end{proof}

The following two lemmas are crucial in our considerations. The first one is due to Borsik.

\begin{lemma}[Borsik, {\cite[Lemma 1]{Borsik}}]
\label{lem2}
Let $X$ be a metric space. Suppose that $N\subset X$ is a non-empty closed nowhere dense set, and $U\subset X$ is semi-open with $N\subset\overline{U}$. Then there is a sequence of pairwise disjoint non-empty semi-open sets $(G_n)_{n\in\omega}$, such that $\bigcup_{n\in\omega}G_n=U\setminus N$ and $N\subset\overline{G_n}$ for each $n\in\omega$.
\end{lemma}

\begin{corollary}
\label{wn}
Let $X$ be a metric space. Suppose that $N, M\subset X$ are non-empty closed nowhere dense sets with $N\subset M$, and $G\subset X$ is semi-open with $M\subset\overline{G}$. Then there are two disjoint non-empty semi-open sets $V$ and $W$ such that:
\begin{itemize}
	\item $N\subset\overline{V}$;
	\item $M\subset\overline{W}$;
	\item $V\subset G\setminus M$;
	\item $V\cup W=G\setminus N$.
\end{itemize}
\end{corollary}

\begin{proof}
Apply Lemma \ref{lem2} to $M$ and $G$ to get a sequence of pairwise disjoint non-empty semi-open sets $(G_n)_{n\in\omega}$ such that $\bigcup_{n\in\omega}G_n=G\setminus M$ and $M\subset\overline{G_n}$ for each $n\in\omega$. Let $V=G_0$ and $W=\bigcup_{n>0}G_n\cup (M\setminus N)$. Note that $W$ is semi-open. Then $V$ and $W$ are as needed.
\end{proof}

\begin{lemma}
\label{lem3}
Let $X$ be a metric space. Suppose that $\I$ is an ideal on $\omega$ such that there are a partition $(A_n)_{n\in\omega}$ of $\omega$ and a function $\phi\colon \omega\to\omega$ satisfying:
\begin{itemize}
	\item[(a)] $\phi (p)>k$ for all $p\in A_k$ and $k\in\omega$;
	\item[(b)] $\left(\forall_{n\in \omega}\ p_{n+1}\in\bigcup_{i\geq\phi(p_n)}A_i\right)\Rightarrow \{p_n:n\in\omega\}\in\I$ for any $(p_n)_{n\in\omega}\subset \omega$.
\end{itemize}
Then $PWD\left(X\right)\subset (\mathcal{I},\mathcal{J})\left(QC\left(X\right)\right)$, where $\J$ is the ideal generated by $(A_n)_{n\in\omega}$.
\end{lemma}

\begin{proof}
Fix any pointwise discontinuous function $f\in\mathbb{R}^X$. Define $\varepsilon_i=\frac{1}{k+1}$ for all $i\in A_k$ and $k\in\omega$. It is easy to see that $(\varepsilon_i)_{i\in \omega}$ is $\J$-convergent to $0$. For each $k\in\omega$ apply Lemma \ref{lem1} to $f$ and $\varepsilon=\frac{1}{k+1}$ to get $N_k\subset X$ and $g_k\colon X\setminus N_k\to\mathbb{R}$ with the required properties. Without loss of generality we can assume that $\emptyset\neq N_0\subset N_1\subset \ldots$. Let also $\{q_n:\ n\in\omega\}$ be an enumeration of $\mathbb{Q}$.

In order to define a sequence of functions $(f_n)_{n\in \omega}$ which $(\I,\J)$-converges to $f$, we need to inductively construct auxiliary semi-open sets $G_{n,m}^k$, $V_{n,m}^k$ and $W_{n,m}^k$ for all $k,m\in\omega$ and $n\in A_k$. 

The induction is on $k$. We start with $k=0$.
\begin{itemize}
	\item Apply Lemma \ref{lem2} to $N_0$ and the semi-open set $U_0=X$ to get non-empty pairwise disjoint semi-open sets $G_{n,m}^0$ for all $n\in A_0$ and $m\in\omega$. 
	\item For each $n\in A_0$ and $m\in\omega$ apply Corollary \ref{wn} to $N_0$, $\overline{G_{n,m}^0}\cap N_{\phi(n)}$ (note that this set is closed and nowhere dense) and $G_{n,m}^0$ to get two disjoint non-empty semi-open sets $W_{n,m}^0$ and $V_{n,m}^0$. 
\end{itemize}

Suppose now that $G_{n,m}^j$'s, $V_{n,m}^j$'s and $W_{n,m}^j$'s for all $m\in\omega$, $n\in A_k$ and $j\leq k$ are already defined. Let
$$U_{k+1}=X\setminus\bigcup_{j\leq k}\bigcup_{m\in\omega}\bigcup_{\substack{n\in A_j\\ \phi(n)> k}}V_{n,m}^j.$$
Note that $N_{k+1}\subset U_{k+1}\subset\overline{U_{k+1}}$. Indeed, if there would be $x\in N_{k+1}\cap V_{n,m}^j$ for some $j\leq k$, $m\in\omega$ and $n\in A_j$ with $\phi(n)> k$, then $x\in\overline{G_{n,m}^j}\cap N_{\phi(n)}$, but this set is disjoint with $V_{n,m}^j$ (cf. Corollary \ref{wn}). Moreover, $U_{k+1}$ is semi-open as a union of semi-open sets:
$$U_{k+1}=\bigcup_{m\in\omega}\bigcup_{n\in A_{k}}\left(W_{n,m}^k\cup N_k\right)\cup\bigcup_{j<k}\bigcup_{m\in\omega}\bigcup_{\substack{n\in A_j\\ \phi(n)=k}}V_{n,m}^j$$
(the sets $W_{n,m}^k\cup N_k$ are semi-open, since $N_k\subset\overline{G_{n,m}^{k}}\cap N_{\phi(n)}\subset \overline{W_{n,m}^k}\subset\overline{{\rm int}(W_{n,m}^k\cup N_k)}$ for each $m\in\omega$ and $n\in A_{k}$).
\begin{itemize}
	\item Apply Lemma \ref{lem2} to $N_{k+1}$ and $U_{k+1}$ to get non-empty pairwise disjoint semi-open sets $G_{n,m}^{k+1}$ for all $n\in A_{k+1}$ and $m\in\omega$. 
	\item For each $n\in A_{k+1}$ and $m\in\omega$ apply Corollary \ref{wn} to $N_{k+1}$, $\overline{G_{n,m}^{k+1}}\cap N_{\phi(n)}$ and $G_{n,m}^{k+1}$ to get two disjoint non-empty semi-open sets $W_{n,m}^{k+1}$ and $V_{n,m}^{k+1}$. 
\end{itemize}

Now we proceed to the construction of $f_n$'s. Set any $n\in \omega$ and let $k$ be such that $n\in A_k$. Define $f_n\colon X\to\mathbb{R}$ by
$$f_n\left(x\right)=\left\{\begin{array}{ll}
f(x) & \mbox{\boldmath{if }} x\in N_k,\\
q_m & \mbox{\boldmath{if }} x\in V_{n,m}^k,\\
g_k(x) & \mbox{\boldmath{otherwise.}}\\
\end{array}\right.$$

We will show that $f_n$ is quasi-continuous. Take any $x\in X$, $\varepsilon>0$ and an open set $W\ni x$. There are three possible cases:
\begin{itemize}
	\item If $x\in N_k$, then there is $m\in\omega$ with $q_m\in (f_n(x)-\varepsilon,f_n(x)+\varepsilon)$. Since $\overline{V_{n,m}^k}\supset N_k$, the set $V_{n,m}^k\cap W$ is non-empty and semi-open. Hence, $W'=\textrm{int} (V_{n,m}^k\cap W)\neq\emptyset$ and $|f_n(x')-f_n(x)|<\varepsilon$ for each $x'\in W'$.
	\item If there is $m\in\omega$ such that $x\in V_{n,m}^k$, then $W'={\rm int}(V_{n,m}^k\cap W)\neq\emptyset$ and $f_n(x')=f_n(x)$ for each $x'\in W'$.
	\item If $x\in X\setminus (N_k\cup\bigcup_{m\in\omega}V_{n,m}^k)$, then $f_n(x)=g_k(x)$ and, by continuity of $g_k$, there is an open neighbourhood $W'\subset W$ of $x$ such that $|f_n(x)-g_k(x')|<\varepsilon$ for all $x'\in W'$. There is also a semi-open set $H$ containing $x$ ($H$ is either one of the $W_{n,m}^k$'s for $m\in\omega$ or one of the $G_{l,m}^k$'s for $l\in A_k\setminus\{n\}$ and $m\in\omega$, or one of the $V_{l,m}^j$'s for $j<k$, $l\in A_j$ with $\phi(l)>k$ and $m\in\omega$). Then, similarly as above, $W''={\rm int}(H\cap W')\neq\emptyset$ and $|f_n(x)-f_n(x')|<\varepsilon$ for each $x'\in W''$ since $f_n\upharpoonright W''=g_k\upharpoonright W''$.
\end{itemize}

Since all $f_n$'s are defined, we are ready to prove that $(f_n)_{n\in\omega}\xrightarrow{(\I,\J)-e} f$. Fix any $x\in X$ and denote
$$P_x=\bigcup_{k\in\omega}\{n\in A_k:\ x\in V^k_{n,m}\textrm{ for some }m\in\omega\}.$$
Observe that $\{n\in\omega:\ |f_n(x)-f(x)|\geq\varepsilon_n\}\subset P_x$. Hence, it suffices to show that $P_x\in\I$.

Given $k\in\omega$, the sets $V^k_{n,m}$ for $n\in A_{k}$, $m\in\omega$ are pairwise disjoint, so $|\{n\in A_k: x\in V^k_{n,m}\textrm{ for some }m\in\omega\}|\leq 1$. If $P_x$ is finite, then we are done, so suppose that it is infinite and let $\{p_0,p_1,\ldots\}$ be an enumeration of the set $P_x$ such that $k(i+1)>k(i)$ for all $i\in\omega$, where $k(i)$ is defined by $p_i\in A_{k(i)}$.

We will use the condition (b). Fix some $i\in\omega$. If $x\in V^{k(i)}_{p_i,m}$ for some $m\in\omega$, then $x\notin V^k_{n',m'}$ for all $k(i)<k<\phi(p_i)$, $n'\in A_k$ and $m'\in\omega$ (since $U_k\cap  V^{k(i)}_{p_i,m}=\emptyset$ and $V^k_{n',m'}\subset U_k$). Therefore, $p_{i+1}\in \bigcup_{j\geq\phi(p_i)}A_j$. Now it follows from the condition (b) that $P_x\in\I$. This finishes the entire proof.
\end{proof}

Now we proceed to the main aim of this subsection.

\begin{proposition}
\label{q-type 2.2}
Suppose that $\I$ and $\J$ are ideals on $\omega$ such that there is a sequence $(A_n)_{n\in\omega}$ of elements of $\J$ with $\mathcal{I}\sqcup (A_n)_{n\in\omega}$ not weakly Ramsey. Then we have $PWD\left(X\right)\subset (\mathcal{I},\mathcal{J})\left(QC\left(X\right)\right)$ for any metric Baire space $X$.
\end{proposition}

\begin{proof}
Let $\left(A_{n}\right)_{n\in\omega}\subseteq \J$ be such that $\mathcal{WR}\sqsubseteq \mathcal{I}\sqcup (A_n)_{n\in\omega}$ (cf. Theorem \ref{WR}). There is a bijection $\pi\colon\omega\to\omega^2$ with $\pi^{-1} [M]\in\mathcal{I}\sqcup (A_n)_{n\in\omega}$ for any $M\in\mathcal{WR}$ (cf. Theorem \ref{WR}). Let $\pi_1,\pi_2\colon\omega\to\omega$ be given by $\pi(x)=(\pi_1(x),\pi_2(x))$ for all $x\in\omega$.

Without loss of generality we can assume that $(A_n)_{n\in\omega}$ is a partition of $\omega$. If there is $A\in \J$ such that $\mathcal{WR}\sqsubseteq \mathcal{I}\sqcup A$, then we are done by Theorem \ref{q-type 3}. Suppose that $\mathcal{I}\sqcup A$ does not contain an isomorphic copy of $\mathcal{WR}$ for any $A\in \J$. Then we can assume that $(A_n)_{n\in\omega}\subseteq \I^+$. 

For each $k\in\omega$ there exist $N_k\in\omega$ and disjoint sets $B_{k}$ and $C_{k}$ such that $\pi^{-1} [\{k\}\times\omega]=B_{k}\cup C_{k}$, $B_{k}=\bigcup_{n\leq N_{k}}A_{n}\cap \pi^{-1} [\{k\}\times\omega]$ and $C_{k}\in \I$. Assume additionally that $N_{0}< N_{1}<\ldots$ (in particular, $N_k\geq k$).

Denote $B=\bigcup_{k\in\omega}B_k$ and $C=\bigcup_{k\in\omega}C_k$. Then $B\cup C=\omega$ and, by Lemma \ref{4}, it suffices to prove that $PWD\left(X\right)\subset(\mathcal{I}\upharpoonright Z,\mathcal{J}\upharpoonright Z)\left(QC\left(X\right)\right)$ for $Z=B,C$. 

{\bf The set $B$. }Note that $(A_n\cap B)_{n\in\omega}$ is a partition of $B$ into sets belonging to $\J\upharpoonright B$. Consider $\phi_B\colon B\to\omega$ given by 
$$\phi_B(p)=\min\{i>m:\ \forall_{k\leq \pi_1(p)+\pi_2(p)}\forall_{j\geq i}\  A_j\cap B_k=\emptyset\},$$
where $m$ is such that $p\in A_m\cap B$. Observe that $\phi_B$ is well defined and
\begin{equation}
i\geq\phi_B(p)\Rightarrow A_i\cap B\subset\bigcup\{B_k:\  k>\pi_1(p)+\pi_2(p)\}. \label{eq:0}
\end{equation}

We will show that $(A_n\cap B)_{n\in\omega}$, $\phi_B$ and $\I\upharpoonright B$ satisfy conditions (a) and (b) of Lemma \ref{lem3}. It will follow that $PWD\left(X\right)\subset (\mathcal{I}\upharpoonright B,\mathcal{J}\upharpoonright B)\left(QC\left(X\right)\right)$ for any metric space $X$.

The condition (a) is obvious. To show the condition (b), take any $(p_n)_{n\in\omega}\subset B$ with $p_{n+1}\in\bigcup_{i\geq\phi_B(p_n)}A_i\cap B$ for all $n\in \omega$ and denote $P=\{p_n:\ n\in\omega\}$. 

Firstly, observe that $\pi[P]\in\mathcal{WR}$, since $\pi(p_{n+1})$ belongs to
$$\pi\left[\bigcup_{i\geq\phi_B(p_n)}A_i\cap B\right] \subset\pi\left[\bigcup\{B_k: k>\pi_1(p_{n})+\pi_2(p_{n})\}\right]\subset (\omega\setminus (\pi_1(p_{n})+\pi_2(p_{n})))\times\omega$$
by (\ref{eq:0}). Hence, $P\in\left(\mathcal{I}\sqcup (A_n)_{n\in\omega}\right)\upharpoonright B$. What is more, $|P\cap A_i|\leq 1$ for all $i\in\omega$, by the condition (a). Therefore, $P\in\I\upharpoonright B$.  

{\bf The set $C$. }Observe that $A_i\cap C\subset\bigcup_{k\leq i}C_k$ for all $i\in\omega$. Indeed, if $i<k$, then $i<N_k$ and $A_i\cap \pi^{-1} [\{k\}\times\omega]\subset B_k$, hence, $A_i\cap C_k=\emptyset$. Recall that each $C_k$ is in $\I$. Hence, $A_i\cap C\in\mathcal{I}\upharpoonright C$ for all $i\in\omega$. Therefore, $\mathcal{I}\sqcup (A_n)_{n\in\omega}\upharpoonright C=\I\upharpoonright C$.  By Fact \ref{monotonicznosc}, the ideal $\mathcal{I}\sqcup (A_n)_{n\in\omega}\upharpoonright C$ is not weakly Ramsey. It follows that $\I\upharpoonright C$ is not weakly Ramsey. By Fact \ref{zawieranie}, the pair $(\I\upharpoonright C,\fin\upharpoonright C)$ is of the third q-type. Then $PWD\left(X\right)\subset \textrm{Baire}\left(X\right)\subset (\mathcal{I}\upharpoonright C,\fin\upharpoonright C)\left(QC\left(X\right)\right)\subseteq  (\mathcal{I}\upharpoonright C,\mathcal{J}\upharpoonright C)\left(QC\left(X\right)\right)$ for any metric Baire space $X$ by Proposition \ref{q-type 3}.
\end{proof}

\subsection{Definable ideals}

We are ready to prove the main theorems of this section, summarizing all of our previous considerations.

\begin{theorem}
\label{Borel1}
Let $\I$ and $\J$ be non-orthogonal ideals on $\omega$. Suppose that $\I$ is coanalytic. 
\begin{enumerate}
	\item $(\mathcal{I},\mathcal{J})$ is of the first q-type if and only if $(\mathcal{I},\mathcal{J})\left(QC\left(X\right)\right)=PWD_0\left(X\right)$ for every metric Baire space $X$.
	\item $(\mathcal{I},\mathcal{J})$ is of the second q-type if and only if $(\mathcal{I},\mathcal{J})\left(QC\left(X\right)\right)=PWD\left(X\right)$ for every metric Baire space $X$.
	\item $(\mathcal{I},\mathcal{J})$ is of the third q-type if and only if $\textrm{Baire}\left(X\right)\subset(\mathcal{I},\mathcal{J})\left(QC\left(X\right)\right)$ for every metric Baire space $X$. Moreover, if $\I$ is Borel, then $\textrm{Baire}\left(X\right)=(\mathcal{I},\mathcal{J})\left(QC\left(X\right)\right)$ for every metric Baire space $X$.
\end{enumerate}
\end{theorem}

\begin{proof}
Since $\I$ is coanalytic, by Fact \ref{det}, each pair $(\mathcal{I},\mathcal{J})$ is of some q-type. Therefore, in parts (1), (2) and (3) it suffices to prove only the implication from left to right, since the classes $PWD_0\left(\mathbb{R}\right)$, $PWD\left(\mathbb{R}\right)$ and $\textrm{Baire}\left(\mathbb{R}\right)$ do not coincide.

{\bf Part (1): }This is exactly Proposition \ref{q-type 1}.

{\bf Part (2): }The inclusion "$\subset$" follows from Proposition \ref{q-type 2.1} and the opposite one -- from Proposition \ref{q-type 2.2}.

{\bf Part (3): }The inclusion "$\supset$" is exactly Proposition \ref{q-type 3}. To prove the opposite one in the case of $\I$ being Borel, consider a sequence $(f_n)_{n\in\omega}\subset\R^{X}$ of quasi-continuous functions such that $(f_n)_{n\in\omega}\xrightarrow{(\mathcal{I},\mathcal{J})-e} f$ for some $f\in \R^{X}$. By Lemma \ref{3}, $(f_n)_{n\in\omega}\xrightarrow{\mathcal{I}\sqcup (A_n)_{n\in\omega}} f$ for some $(A_n)_{n\in\omega}\subset\J$. The ideal $\mathcal{I}\sqcup (A_n)_{n\in\omega}$ is Borel by Lemma \ref{1}. Now it follows from Lemma \ref{0} that $f\in \textrm{Baire}(X)$.
\end{proof}

\begin{remark}
The implications from left to right in parts (1), (2) and (3) of Theorem \ref{Borel1} remain true even if we drop the assumption that $\I$ is coanalytic.
\end{remark}

The next result characterizes higher Baire classes (generated by quasi-continuous functions) with respect to $(\I,\J)$-equal convergence.

\begin{proposition}
Suppose that $\I$ and $\J$ are non-orthogonal ideals on $\omega$. Then the classes $(\mathcal{I},\mathcal{J})\left(PWD_0\left(X\right)\right)$, $(\mathcal{I},\mathcal{J})\left(PWD\left(X\right)\right)$ and $(\mathcal{I},\mathcal{J})\left(\textrm{Baire}\left(X\right)\right)$ all contain the class $\textrm{Baire}(X)$ for every metric Baire space $X$. Moreover, if $\I$ is analytic, then all those classes are equal to $\textrm{Baire}(X)$ for every metric Baire space $X$. 
\end{proposition}

\begin{proof}
Since $PWD_0\left(X\right)\subset PWD\left(X\right)\subset \textrm{Baire}\left(X\right)$, we have: $$(\mathcal{I},\mathcal{J})\left(PWD_0\left(X\right)\right)\subset(\mathcal{I},\mathcal{J})\left(PWD\left(X\right)\right)\subset(\mathcal{I},\mathcal{J})\left(\textrm{Baire}\left(X\right)\right).$$ 
By \cite[Theorem 9 and Proposition 16]{nat-szuca2}, for every Baire function $f\in\R^{X}$ there is a sequence of functions in $PWD_0(X)$ which discretely converges to $f$. Now the inclusion $\textrm{Baire} (X)\subset (\mathcal{I},\mathcal{J})\left(PWD_0\left(X\right)\right)$ follows from Lemma \ref{2}. Finally, if $\I$ is analytic, then the inclusion $(\mathcal{I},\mathcal{J})\left(\textrm{Baire}\left(X\right)\right)\subset \textrm{Baire}(X)$ follows from Lemma \ref{0} similarly as in part (3) of the previous Theorem. 
\end{proof}

\section{Ideal equal convergence of sequences of continuous functions}
\label{Ideal equal convergence of sequences of continuous functions}

In this section we want to characterize ideal equal Baire classes generated by the family of continuous functions. These studies extend the results from \cite{filipow-szuca-I-Baire}. In the first subsection we introduce some useful notions. Next, we prove the mentioned characterization.

\subsection{An infinite game and the c-types}

Let $\I$ be an ideal on $\omega$. Consider another game, $G_{2}\left(\I\right)$, defined by Laflamme (see \cite{laf}) as follows: Player I in his $n$'th move plays an element $C_{n}\in\I$, and then Player II responses with any $F_{n}\in \left[\omega\right]^{<\omega}$ such that $F_{n}\cap C_{n}=\emptyset$. Player I wins if $\bigcup_{n\in \omega} F_{n}\in \I$. Otherwise, Player II wins.

\begin{theorem}[{\cite[Fact 3.10]{KZ}}]
\label{determinacja2}
If $\mathcal{I}$ is coanalytic, then the game $G_{2}(\mathcal{I})$ is determined.
\end{theorem}

A set $\mathcal{Z}=\left\{A_{m}:\ m\in \omega\right\}\subseteq \left[\omega\right]^{<\omega}\setminus \left\{\emptyset\right\}$ is \emph{$\I^*$-universal} if for each $F\in \I^*$ there is $m\in \omega$ such that $A_{m}\subseteq F$. We say that $\I$ is \emph{$\omega$-diagonalizable by $\mathcal{I}^*$-universal sets} if there exists a sequence $(\mathcal{Z}_{N})_{N\in \omega}$ of $\I^*$-universal sets such that for each $F\in \I^*$ there is $\mathcal{Z}_{N}=\left\{A_{N,m}:\ m\in \omega\right\}$ with $A_{N,m}\cap F\neq \emptyset$ for every $m\in \omega$. An ideal $\I$ is a \emph{weak $P$-ideal} if for every sequence $(X_{n})_{n\in \omega}\subseteq \I$ there exists $X\in \I^+$ such that $X_{n}\cap X\in \fin$ for every $n\in \omega$. The above notions were introduced by Laflamme in order to give the following characterization.

\begin{theorem}[Laflamme, {\cite[Theorem 2.16]{laf}}]\label{234444}
Let $\mathcal{I}$ be an ideal.
\begin{enumerate}
	\item Player I has a winning strategy in $G_{2}(\mathcal{I})$ if and only if $\I$ is not a weak $P$-ideal.	
	\item Player II has a winning strategy in $G_{2}(\mathcal{I})$ if and only if $\I$ is $\omega$-diagonalizable by $\mathcal{I}^*$-universal sets.
\end{enumerate}
\end{theorem}

\begin{theorem}[\cite{LR} and \cite{barb}]
\label{finfin}
The following are equivalent for any ideal $\mathcal{I}$:
\begin{enumerate}
\item $\I$ is not a weak $P$-ideal;
\item $\fin\otimes\fin\sqsubseteq\mathcal{I}$;
\item $\fin\otimes\fin\leq_{K}\mathcal{I}$.
\end{enumerate}
\end{theorem}

It follows from the above theorems that if $\I$ is a coanalytic ideal, then either $\fin\otimes\fin\sqsubseteq\mathcal{I}$ or $\I$ is $\omega$-diagonalizable by $\mathcal{I}^*$-universal sets.

Analogously to the q-types, we define the c-types of pairs of ideals.

\begin{definition}
Let $\I$ and $\J$ be ideals on $\omega$.
\begin{enumerate}
	\item $(\I,\J)$ is of the \emph{first c-type} if for any sequence $(A_n)_{n\in\omega}$ of elements of $\J$ the ideal $\mathcal{I}\sqcup (A_n)_{n\in\omega}$ is $\omega$-diagonalizable by $(\mathcal{I}\sqcup (A_n)_{n\in\omega})^*$-universal sets.
	\item $(\I,\J)$ is of the \emph{second c-type} if there is a sequence $(A_n)_{n\in\omega}$ of elements of $\J$ such that the ideal $\mathcal{I}\sqcup (A_n)_{n\in\omega}$ contains an isomorphic copy of $\fin\otimes\fin$, but for any $A\in\J$ the ideal $\mathcal{I}\sqcup A$ is $\omega$-diagonalizable by $(\mathcal{I}\sqcup A)^*$-universal sets.
	\item $(\I,\J)$ is of the \emph{third c-type} if there is $A\in\J$ such that the ideal $\mathcal{I}\sqcup A$ contains an isomorphic copy of $\fin\otimes\fin$.
\end{enumerate}
\end{definition}

\begin{fact}
\label{det2}
If $\I$ is coanalytic, then each pair $(\I,\J)$ is of some c-type.
\end{fact}

\begin{proof}
It follows from Lemma \ref{1} that $\mathcal{I}\sqcup (A_n)_{n\in\omega}$ and $\mathcal{I}\sqcup A$ are coanalytic for any $(A_n)_{n\in\omega}$ and $A\subset\omega$. Then we are done by Theorems \ref{determinacja2}, \ref{234444} and \ref{finfin}.
\end{proof}

Examples of pairs of ideals for every c-type are similar to the examples of pairs of ideals for every q-type from the previous section.

\subsection{The first and third c-type}

In this subsection we characterize $(\mathcal{I},\mathcal{J})\left(C\left(X\right)\right)$ for all pairs of ideals $(\mathcal{I},\mathcal{J})$ of the first or third c-type.

\begin{proposition}\label{a}
Let $X$ be a perfectly normal topological space and $1\leq n<\omega$. Suppose that $\I$ and $\J$ are ideals on $\omega$ such that for any sequence $(A_n)_{n\in\omega}$ of elements of $\J$ the ideal $\mathcal{I}\sqcup (A_n)_{n\in\omega}$ is $\omega$-diagonalizable by $(\mathcal{I}\sqcup (A_n)_{n\in\omega})^*$-universal sets. Then $(\mathcal{\fin},\mathcal{\fin})_{n}\left(C\left(X\right)\right)= (\mathcal{I},\mathcal{J})_{n}\left(C\left(X\right)\right)$.
\end{proposition}

\begin{proof}
The proof is the same as the proof of \cite[Theorem 5.5]{filipow-szuca-I-Baire}.
\end{proof}

\begin{lemma}[Filip{\'o}w and Szuca, {\cite[Lemma 2.2]{filipow-szuca-I-Baire}}]\label{nbvcxz}
Let $X$ be a topological space, $\I$ be an ideal such that $\fin\otimes\fin\leq_{K}\mathcal{I}$ and $1\leq \alpha<\omega_{1}$. Then $(\fin,\fin)_{\alpha+1}\left(C\left(X\right)\right)\subseteq (\I,\fin)_{\alpha}\left(C\left(X\right)\right)$.
\end{lemma}

\begin{proposition}\label{asdf3}
Let $X$ be a topological space and $1\leq \alpha<\omega_{1}$. Suppose that $\I$ and $\J$ are ideals on $\omega$ such that there exists $A\in \J$ with $\fin\otimes\fin\sqsubseteq\mathcal{I}\sqcup A$. Then $(\mathcal{\fin},\mathcal{\fin})_{\alpha+1}\left(C\left(X\right)\right)\subseteq (\mathcal{I},\mathcal{J})_{\alpha}\left(C\left(X\right)\right)$.
\end{proposition}

\begin{proof}
Let $A\in\J$ be such that $\fin\otimes\fin\sqsubseteq\mathcal{I}\sqcup A$. Then $\fin\otimes\fin\leq_{K}\mathcal{I}\upharpoonright (\omega\setminus A)$ and $$(\fin,\fin)_{\alpha+1}\left(C\left(X\right)\right)\subseteq (\I\upharpoonright (\omega\setminus A),\fin (\omega\setminus A))_{\alpha}\left(C\left(X\right)\right)$$
by Lemma \ref{nbvcxz}. It follows that
$$(\fin,\fin)_{\alpha+1}\left(C\left(X\right)\right)\subseteq (\I\upharpoonright (\omega\setminus A),\J\upharpoonright (\omega\setminus A))_{\alpha}\left(C\left(X\right)\right).$$

Since the ideals $\mathcal{I}\upharpoonright A$ and $\mathcal{J}\upharpoonright A=\mathcal{P}(A)$ are orthogonal, we have $$(\mathcal{\fin},\mathcal{\fin})_{\alpha+1}\left(C\left(X\right)\right)\subset\mathbb{R}^X\subset (\mathcal{I}\upharpoonright A,\mathcal{J}\upharpoonright A)_{\alpha}\left(C\left(X\right)\right)$$ by Lemma \ref{5}. The conclusion follows from Lemma \ref{4}.
\end{proof}

\subsection{The second c-type}

In this subsection we characterize $(\mathcal{I},\mathcal{J})\left(C\left(X\right)\right)$ for all pairs of ideals $(\mathcal{I},\mathcal{J})$ of the second c-type.

Let $\Sigma_{\alpha}^{0}\left(X\right)$ and $\Pi_{\alpha}^{0}\left(X\right)$, for $0<\alpha<\omega_{1}$, denote the additive and multiplicative Borel classes of subsets of $X$, respectively. 

\begin{lemma}[{\cite[Proposition 3.14]{csaszar-laczkovich-1979}}]\label{jkgfs}
Let $X$ be a perfectly normal topological space, $f:X\rightarrow \mathbb{R}$ and $1\leq \alpha<\omega_{1}$. Then $f$ is of Baire class $\alpha$ if and only if $f$ is $\Sigma_{\alpha+1}^{0}\left(X\right)$-measurable. 
\end{lemma}

\begin{proposition}\label{asdf2}
Let $X$ be a perfectly normal topological space. Suppose that $\I$ and $\J$ are ideals on $\omega$ such that $\left(\mathcal{I}\sqcup A\right)$ is $\omega$-diagonalizable by $(\mathcal{I}\sqcup A)^*$-universal sets for every $A\in \J$. Then $(\I,\J)_{\alpha}\left(C(X)\right)\subseteq  B_{\alpha}\left(X\right)$ for every $1\leq \alpha<\omega_{1}$.
\end{proposition}

\begin{proof}
This proof is based on the proof of \cite[Lemma 3.1]{filipow-szuca-I-Baire}. 

We prove the result by transfinite induction on $\alpha$. Let $1\leq \alpha<\omega_{1}$ and assume that $(\I,\J)_{\gamma}\left(C(X)\right)\subseteq  B_{\gamma}\left(X\right)$ for every $\gamma< \alpha$. Suppose that $(f_n)_{n\in\omega}\xrightarrow{(\mathcal{I},\mathcal{J})-e} f$, where $f_{n}\in (\I,\J)_{\beta_{n}}\left(C(X)\right)$ and $\beta_{n}<\alpha$ for each $n\in \omega$. Then there exists a sequence $(\varepsilon_n)_{n\in\omega}\xrightarrow{\J}0$ such that $\left\{n\in \omega: |f_{n}(x)-f(x)|\geq\varepsilon_{n}\right\}\in \I$ for every $x\in X$. 

We need to show that $f\in B_{\alpha}\left(X\right)$. Let $\varepsilon >0$, $y\in \R$ and $x\in X$. The conclusion will follow from the fact that $f^{-1}[(y-\varepsilon,y+\varepsilon)]\in \Sigma_{\alpha+1}^{0}\left(X\right)$ for any $\varepsilon >0$ and $y\in \R$ (by Lemma \ref{jkgfs}). Hence, let $\varepsilon >0$ and $y\in \R$.

Define $A_0=\{k\in\omega:\ \varepsilon_k\geq\varepsilon\}$ and $A_n=\{k\in\omega:\ \frac{\varepsilon}{n+1}\leq \varepsilon_k <\frac{\varepsilon}{n}\}$ for all $n\geq 1$. Clearly, $(A_{n})_{n\in \omega}\subseteq \J$. For each $n\in \omega$ pick a family $(\mathcal{Z}_{N}^{n})_{N\in \omega}$, $\mathcal{Z}_{N}^{n}=\left\{A_{N,k}^{n}:\ k\in \omega\right\}$, of $\left(\I\sqcup (A_{0}\cup ...\cup A_{n})\right)^*$-universal sets which $\omega$-diagonalize $\I\sqcup (A_{0}\cup ...\cup A_{n})$.

We will show that 
\begin{equation}
|f\left(x\right)-y|<\varepsilon\Longleftrightarrow \exists_{n\in \omega} \exists_{N\in \omega} \forall_{k\in \omega}\exists_{l\in A_{N,k}^{n}}\ |f_{l}\left(x\right)-y|\leq\varepsilon\cdot \left(1-\frac{1}{n}\right).
\label{eq:1}
\end{equation}
This will end the proof. Indeed, once this is done, we have
$$\displaystyle f^{-1}\left[B\left(y,\varepsilon\right)\right]=\bigcup_{n\in \omega}\bigcup_{N\in \omega}\bigcap_{k\in \omega}\bigcup_{l\in A_{N,k}^{n}}f^{-1}_{l}\left[\overline{B} \left(y,\varepsilon\cdot \left(1-\frac{1}{n}\right)\right)\right]\in \Sigma_{\alpha+1}^{0}\left(X\right),$$ 
where $B(z,r)$ denotes the open ball of radius $r>0$ and center $z\in\R$, by the induction assumption (note that $A_{N,k}^{n}$ is finite).

We proceed to showing (\ref{eq:1}). Firstly, we deal with the implication from left to right. Let $f\left(x\right)\in B\left(y,\varepsilon\right)$. There are $n_{1}\in \omega$ and $\delta>0$ such that $B\left(f\left(x\right),\delta\right)\subseteq \overline{B} \left(y,\varepsilon\cdot \left(1-\frac{1}{n_{1}}\right)\right)$. Take $n>n_{1}$ such that $\frac{\varepsilon}{n}<\delta$ and denote
$$F=\left\{l\in \omega:\ f_{l}\left(x\right)\in \overline{B} \left(y,\varepsilon\cdot \left(1-\frac{1}{n}\right)\right)\right\}.$$ 
Then
$$F\supseteq \left\{l\in \omega: f_{l}\left(x\right)\in B\left(f\left(x\right),\delta\right)\right\}\in \left(\I\sqcup (A_{0}\cup ...\cup A_{n-1})\right)^*.$$
Hence, there is $N\in \omega$ such that $F\cap A_{N,k}^{n-1}\neq \emptyset$ for every $k\in \omega$ (since $(\mathcal{Z}_{N}^{n-1})_{N\in \omega}$ $\omega$-diagonalize $\I\sqcup (A_{0}\cup ...\cup A_{n-1})$).

Now we deal with the second implication of (\ref{eq:1}).  Suppose that there are $n,N\in \omega$ such that for every $k\in \omega$ there is $l\in A_{N,k}^{n}$ with $f_{l}\left(x\right)\in \overline{B} \left(y,\varepsilon\cdot \left(1-\frac{1}{n}\right)\right)$. Observe that $G=\left\{m\in \omega:\ |f_{m}(x)-f(x)|<\frac{\varepsilon}{n}\right\}\in \left(\I\sqcup (A_{0}\cup ...\cup A_{n})\right)^*$. Since $\mathcal{Z}_{N}^{n}$ is $\left(\I\sqcup (A_{0}\cup ...\cup A_{n})\right)^*$-universal, there is $k\in \omega$ such that $A_{N,k}^{n}\subseteq G$. By our assumption, there is also $l\in A_{N,k}^{n}$ such that $f_{l}\left(x\right)\in \overline{B} \left(y,\varepsilon\cdot \left(1-\frac{1}{n}\right)\right)$. Then $$|f(x)-y|\leq |f(x)-f_{l}(x)|+|f_{l}(x)-y|<\frac{\varepsilon}{n}+\varepsilon\left(1-\frac{1}{n}\right)=\varepsilon.$$
This finishes the entire proof.
\end{proof}

\begin{proposition}\label{asdf}
Let $X$ be a topological space. Suppose that $\I$ and $\J$ are ideals on $\omega$ such that there exists $(A_{n})_{n\in\omega}\subseteq \J$ with $\fin\otimes\fin\sqsubseteq \mathcal{I}\sqcup (A_n)_{n\in\omega}$. Then
$(\I,\J)_{\alpha}\left(C(X)\right)\supseteq  B_{\alpha}\left(X\right)$ for every $1\leq \alpha<\omega_{1}$.
\end{proposition}

\begin{proof}
We prove the result by transfinite induction on $\alpha$. Let $1\leq \alpha<\omega_{1}$ and assume that $(\I,\J)_{\gamma}\left(C(X)\right)\supseteq  B_{\gamma}\left(X\right)$ for every $\gamma< \alpha$. 

Let $\left(A_{n}\right)_{n\in\omega}\subseteq \J$ be such that $\fin\otimes\fin\sqsubseteq \mathcal{I}\sqcup (A_n)_{n\in\omega}$. Then there is a bijection $\sigma \colon\omega\to\omega^2$ such that $\sigma^{-1} [M]\in\mathcal{I}\sqcup (A_n)_{n\in\omega}$ for any $M\in\fin\otimes\fin$. Without loss of generality we can assume that $(A_n)_{n\in\omega}$ is a partition of $\omega$. 

If there is $A\in \J$ such that $\fin\otimes\fin\sqsubseteq \mathcal{I}\sqcup A$, then we are done by Theorem \ref{asdf3}, since $B_\alpha (X)\subset (\fin,\fin)_{\alpha+1}(C(X))$. Suppose that $\mathcal{I}\sqcup A$ does not contain an isomorphic copy of $\fin\otimes\fin$ for every $A\in \J$. Then we can assume that $(A_n)_{n\in\omega}\subseteq \I^+$. 

For each $k\in \omega$ there are $N_k\in\omega$ and $C_k=\sigma^{-1} [\{k\}\times\omega]\setminus \bigcup_{n\leq N_{k}}A_{n}$ such that $C_{k}\in \I$. Without loss of generality we can assume additionally that $N_{0}< N_{1}<\ldots$ (in particular, $C_k\cap A_n=\emptyset$ whenever $n\leq k$) and $C_{k}=\emptyset$ if $\sigma^{-1} [\{k\}\times\omega]$ can be covered by finitely many $A_n$'s (in particular, each $C_k$ is infinite or empty). 

Define $T=\{k\in\omega:\ C_k\neq\emptyset\}$. Let $G_1=\bigcup_{k\in T}C_k$ and $G_2=\omega\setminus G_1$. We will show that $B_{\alpha}\left(X\right)\subset (\I\upharpoonright G_i,\J\upharpoonright G_i)_{\alpha}\left(C(X)\right)$ for $i=1,2$. It will finish the proof by Lemma \ref{4}.

Firstly, we deal with the set $G_1$. If $T$ is finite, then $G_1\in\I$ and we are done by Lemma \ref{5} (since $\I\upharpoonright G_1=\mathcal{P}(G_1)$ in this case). Suppose that $T$ is infinite. We will prove that $\fin\otimes\fin\leq_K \mathcal{I}\upharpoonright G_1$. Once this is done, we have
$$B_{\alpha}\left(X\right)\subset (\fin,\fin)_{\alpha+1}\left(C(X)\right)\subset (\I\upharpoonright G_i,\fin (G_i))_{\alpha}\left(C(X)\right)$$
by Lemma \ref{nbvcxz}. Hence,
$$B_{\alpha}\left(X\right)\subset (\I\upharpoonright G_i,\J\upharpoonright G_i)_{\alpha}\left(C(X)\right).$$

We claim that $\sigma\upharpoonright G_1\colon G_1\to\omega^2$ witnesses $\fin\otimes\fin\leq_K \mathcal{I}\upharpoonright G_1$. Take any $M\in \fin\otimes\fin$ with $M\subset\sigma [G_1]$. There exist $E\in \fin\otimes\emptyset$ and $F\in \emptyset\otimes\fin$ such that $M=E\cup F$. Since $C_k\subset \sigma^{-1} [\{k\}\times\omega]$ for each $k\in\omega$, we get that $\sigma^{-1}[E]$ is covered by finitely many $C_k$'s. Recall that $C_{k}\in \I$ for all $k\in\omega$. Hence, $\sigma^{-1}[E]$ is in $\I\upharpoonright G_1$. Now we deal with the set $F$. From the properties of $\sigma$ we have that $\sigma^{-1} [F]\in\mathcal{I}\upharpoonright G_1\sqcup (A_n\cap G_1)_{n\in\omega}$. Observe that $\sigma^{-1}[F]\cap A_{n}\subseteq \sigma^{-1}[F]\cap \bigcup_{k<n} C_k$. Indeed, it follows from the fact that $C_k\cap A_n=\emptyset$ whenever $n\leq k$. Moreover, $\sigma^{-1}[F]\cap \bigcup_{k<n} C_k$ is finite, since $F\in \emptyset\otimes\fin$ and $C_k\subset \sigma^{-1} [\{k\}\times\omega]$ for each $k\in\omega$. Therefore, $\sigma^{-1} [F]\in\mathcal{I}\upharpoonright G_1$. 

Now we deal with the set $G_2$. We will need two auxiliary ideals. Define an ideal 
$$\K=\left\{M\subseteq G_{2}:\ \forall_{k\in \omega}\ A_{k}\cap M\in \fin\right\}.$$
Let also $\mathcal{L}$ be an ideal on $G_{2}$ generated by the family $\left(A_{k}\cap G_2\right)_{k\in \omega}$. Recall that by $\wlasnosc(\K,\mathcal{L})$ we denote the following sentence: For every partition $(A_n)_{n\in\omega}\subset\mathcal{L}$ of $\bigcup\mathcal{L}$ there exists $S\notin\K$ such that $A_n\cap S\in\K$ for every $n\in\omega$ (cf. Lemma \ref{wlasn}). Therefore, $\wlasnosc\left(\K,\mathcal{L}\right)$ does not hold.

Fix $f\in B_{\alpha}\left(X\right)$. We will show that $f\in (\I,\J)_{\alpha}\left(C(X)\right)$. There is a sequence of functions in $\bigcup_{\gamma<\alpha}B_{\gamma}\left(X\right)$ which is $\K$-convergent to $f$ (recall that pointwise convergence implies ideal convergence for any ideal). From our induction assumption, this sequence is also in $\bigcup_{\gamma<\alpha} (\I,\J)_{\gamma}\left(C(X)\right)$. Then $f\in (\K,\mathcal{L})\left(\bigcup_{\gamma<\alpha} (\I,\J)_{\gamma}\left(C(X)\right)\right)$ by Lemma \ref{wlasn}, since $\wlasnosc\left(\K,\mathcal{L}\right)$ does not hold. 

Obviously, $\mathcal{L}\subset\J\upharpoonright G_2$. To finish the proof it suffices to show that $\K\subset\I\upharpoonright G_2$. Take $M\in\K$ and notice that $M\cap\sigma^{-1}[\{k\}\times\omega]\subset\bigcup_{i\leq N_k}A_i$ for any $k\in\omega$ (since $M\subset G_2$). Hence, $M\cap\sigma^{-1}[\{k\}\times\omega]$ is finite for every $k\in\omega$. It follows that $\sigma [M]\in \emptyset\otimes\fin$. By the properties of $\sigma$, we get that $M\in \mathcal{I}\upharpoonright G_2\sqcup (A_n\cap G_2)_{n\in\omega}$. Hence, $M\in \I\upharpoonright G_2$ by the definition of $\K$.
\end{proof}

\subsection{Definable ideals}

We are ready to prove the main theorem of this section, summarizing all of our previous considerations.

\begin{theorem}
\label{b}
Let $\I$ and $\J$ be non-orthogonal ideals on $\omega$ and $1\leq n<\omega$. Suppose that $\I$ is coanalytic.
\begin{enumerate}
	\item $(\mathcal{I},\mathcal{J})$ is of the first c-type if and only if $$(\mathcal{I},\mathcal{J})_{n}\left(C\left(X\right)\right)=(\mathcal{\fin},\mathcal{\fin})_{n}\left(C\left(X\right)\right)$$ for every perfectly normal topological space $X$.
	\item $(\mathcal{I},\mathcal{J})$ is of the second c-type if and only if $$(\mathcal{I},\mathcal{J})_{n}\left(C\left(X\right)\right)=B_{n}\left(X\right)$$ for every perfectly normal topological space $X$.
	\item $(\mathcal{I},\mathcal{J})$ is of the third c-type if and only if $$(\mathcal{I},\mathcal{J})_{n}\left(C\left(X\right)\right)\supset(\mathcal{\fin},\mathcal{\fin})_{n+1}\left(C\left(X\right)\right)$$ for every perfectly normal topological space $X$.
\end{enumerate}
\end{theorem}

\begin{proof}
Since $\I$ is coanalytic, the pair $(\mathcal{I},\mathcal{J})$ is of some c-type (by Fact \ref{det2}). Moreover, $(\mathcal{\fin},\mathcal{\fin})_{n}\left(C\left(\R\right)\right)\varsubsetneq B_{n}\left(\R\right)\varsubsetneq  (\mathcal{\fin},\mathcal{\fin})_{n+1}\left(C\left(\R\right)\right)$ for all $1\leq n<\omega$. Therefore, in parts (1), (2) and (3) it suffices to prove only the implication from left to right.

{\bf Part (1): }This is Proposition \ref{a}.

{\bf Part (2): }The inclusion "$\supseteq $" follows from Proposition \ref{asdf} and the opposite one -- from Proposition \ref{asdf2}.

{\bf Part (3): }This is Proposition \ref{asdf3}.
\end{proof}

\begin{remark}
The implications from left to right in parts (1), (2) and (3) of Theorem \ref{b} remain true even if we drop the assumption that $\I$ is coanalytic.
\end{remark}

\begin{remark}
In parts (2) and (3) of Theorem \ref{b} the implications from left to right can be generalized to all $1\leq \alpha<\omega_{1}$. It follows from Propositions \ref{asdf3}, \ref{asdf2}, and \ref{asdf}. 
\end{remark}

Part (3) of the above theorem does not give an exact outcome, i.e., it does not say which class $(\mathcal{I},\mathcal{J})_{n}\left(C\left(X\right)\right)$ is equal to. The case of ideal convergence (not ideal equal convergence) suggests that the answer should depend on some combinatorial properties of the pair $(\I,\J)$ (cf. \cite{debs-ray}). Therefore, the following problem seems to be natural.

\begin{problem}
Characterize $(\mathcal{I},\mathcal{J})\left(C\left(X\right)\right)$ for $(\I,\J)$ of the third c-type. Is it always equal to one of the classes $(\mathcal{\fin},\mathcal{\fin})_{\gamma}\left(C\left(X\right)\right)$ or can it be equal to some $B_{\gamma}\left(X\right)$?
\end{problem}

\subsection*{Acknowledgements}
The first author was supported by the NCN (Polish National Science Centre) grant No. 2012/07/N/ST1/03205. Both authors were supported by the grant BW-538-5100-B858-15.

The authors would like to thank Tomasz Natkaniec for the idea for this article and for presenting them, on seminar in Gda{\'n}sk, the characterization of equal Baire classes generated by quasi-continuous functions in the classical case.

\end{document}